\documentclass{amsart}
\usepackage{amsfonts,amssymb,graphicx}
\usepackage{epsfig,amsmath,latexsym}
\usepackage{amscd, amssymb, amsthm, amsmath,eucal, verbatim}
\usepackage{epstopdf}
\newtheorem{theorem}{Theorem}[section]
\newtheorem{lemma}{Lemma}[section]

\newcommand{\RR}{{\mathbb{R}}}

\newcommand{\EE}{{\mathbb{E}}}
\newcommand{\HH}{{\mathbb{H}}}


\begin{document}
\title{Isoperimetric Regions in Nonpositively Curved Manifolds}
\author{Joel Hass}
\address{Department of Mathematics, University of California, Davis
California 95616 \& School of Mathematics, Institute for Advanced Study, Princeton CA 08540}
\email{hass@math.ucdavis.edu}
\thanks{Partially supported by the Institute for Advanced Study and the  Ambrose Monell Foundation}
\date{\today}
\begin{abstract}
Isoperimetric regions minimize the size of their boundaries among all regions
with the same volume.  In Euclidean and Hyperbolic space, isoperimetric regions
are round balls. We show that isoperimetric regions in two and three-dimensional 
nonpositively curved manifolds are not necessarily convex balls, and need not even be connected.  
\end{abstract}                                     
 \maketitle

\section{Introduction}
A round ball  in Euclidean space $\EE^n$ minimizes the size of its boundary among all regions with the same volume.
The same is true in hyperbolic space $\HH^n$.
In this paper we examine the extent to which this isoperimetric
property carries over to spaces of nonpositive curvature.

Let $|W|$ denote the  volume of a  Riemannian manifold.
An {\em isoperimetric region} in an $n$-manifold is a region $W$ ($n$-dimensional submanifold with $(n-1)$dimensional boundary) such that
 $\partial W$ minimizes $(n-1)$-volume among all regions enclosing the same volume.
So an isoperimetric region satisfies 
$$
|\partial W | = \inf \{| \partial U |  :  U \subset M, \  |U| =|W| \}.
$$

A {\em Cartan--Hadamard} manifold is a complete, simply-connected Riemannian manifold  $(M, g) $ with
nonpositive sectional curvature.
The standard examples are  $\EE^n$, with curvature zero, and  $\HH^n$, with curvature  -1, but more general examples have
variable nonpositive curvature.
Cartan--Hadamard manifolds share many geometric properties
with $\EE^n$ and $\HH^n$. For example, any Cartan--Hadamard manifold is diffeomorphic to $\EE^n$ and
any two points in such a manifold are connected by a unique geodesic.

The classical isoperimetric inequality for $\EE^n$ states that for any region 
 $W \subset \EE^n$, we have $| \partial W| \ge    |\partial B^n|$ where 
 $ B^n  \subset \EE^n$ is a round ball with $| B^n | =| W| $ \cite{Osserman}.
A similar result holds for $\HH^n$, for which the isoperimetric inequality states  that
 $| \partial W| \ge    |\partial B_{-1}^n |$, where $ B_{-1}^n  $ is the
the ball in hyperbolic space satisfying $| B_{-1}^n | = | W| $.

The {\em Generalized Cartan--Hadamard Conjecture} was posed in various forms by 
Aubin, Burago-Zalgaller and by Gromov \cite{Aubin, BuragoZalgaller, Gromov}.
It states that if $M$ is a Cartan--Hadamard
manifold with sectional curvatures bounded above by $\kappa \le 0,$ 
and if $W$ is any region in $M$, then
$| \partial W| \ge   |\partial B_\kappa^n |$, where  $B_\kappa^n$ is 
the ball in the space form of constant curvature $\kappa$ 
with $ | B_\kappa^n | = | W| $.
A  summary of  known results concerning this conjecture
is found in the recent paper of  Kloeckner and Kuperberg \cite{KloecknerKuperberg:15}.
Various cases were  proved by Weil ($n=2, \kappa = 0$)  \cite{Weil}, 
 Bol ($n=2, \kappa < 0$)  \cite{Bol}, Croke ($n=4, \kappa=0$) \cite{Croke}, Kleiner ($n=3, \kappa \le 0$) \cite{Kleiner},
 Morgan and Johnson (small domains) \cite{MorganJohnson},  Druet (under a scalar curvature hypothesis) \cite{Druet},  and Kloeckner and Kuperberg (under 
 a variety of geometric conditions) \cite{KloecknerKuperberg:15}.

In Euclidean and hyperbolic space, isoperimetric regions that realize equality in the
isoperimetric inequality are uniquely realized by round balls.
Isoperimetric regions were classified for
special classes of Riemannian surfaces by Benjami and Cao \cite{BenjamiCao} and by  Howards, Hutchings  and Morgan \cite{HowardsHutchingsMorgan}.
Some cases in flat 3-manifolds were derived by in Hauswirth, Perez, Romon and Ros \cite{Hauswirth}.  The nature of general
isoperimetric regions remains mysterious, even in
well understood manifolds such as tori and complex hyperbolic space  \cite{Ros}.

Kloeckner and Kuperberg posed the  following question \cite[Question 4.1]{KloecknerKuperberg:15}: \\
{\bf Question.} {\em If $M $ is a Cartan-Hadamard manifold and $W$ minimizes  $|\partial W|$
for some fixed value of $|W|$, then is it convex? Is it a topological ball?  }

We show that the answer to both parts of this question is no in dimensions two and three.
In fact the isoperimetric region need not even be connected. 

\begin{theorem} \label{nonconnected}
There are  Cartan--Hadamard manifolds in dimensions two and three that contain
disconnected  isoperimetric regions.
\end{theorem}

We establish this by constructing explicit Cartan--Hadamard manifolds for which the
 isoperimetric regions appear as islands of zero  curvature in a large
hyperbolic ocean. We give an explicit construction of such manifolds for dimension two in Theorem~\ref{dim2} of Section~\ref{dimensiontwo}  and for dimension three
in Theorem~\ref{dim3} of Section~\ref{dimensionthree}. It is straightforward to generalize the construction to arrange for any isoperimetric region to have at least
$k$ components, for $k$ any positive integer.

The construction in Theorem~\ref{nonconnected} appears to have been overlooked.
One possible reason is that the derivation of the associated isoperimetric inequalities can often be reduced to the study of connected regions.
Consider for example a region $W \subset \EE^2$, for which the isoperimetric inequality states  
\begin{equation}
|W| \le \frac{  |\partial W|^2 }{  4\pi }. \label{isopEqn}
\end{equation}
Suppose that a region $W$ has two components $W_1$ and $ W_2$ whose areas sum to those of $W$, so that  $|W_1| + |W_2| = |W|$.
If Equation~(\ref{isopEqn}) holds for each of $W_1, W_2$, then it also holds for $W$,
since
\begin{equation}
 \label{convexity}
|W| =  |W_1| + |W_2|  \le \frac{  |\partial W_1|^2 } {4\pi} + \frac{ |\partial W_2|^2 }{  4\pi }  
\le \frac{ ( |\partial W_1|  + | \partial W_2|)^2 }{  4\pi } =  \frac{  |\partial W|^2 }{  4\pi }. 
\end{equation} 
The convexity property used to obtain the inequality in Equation~(\ref{convexity}) extends to the
isoperimetric inequalities conjectured to hold in all Cartan--Hadamard spaces \cite{KloecknerKuperberg:15},
so it suffices to establish the  Cartan--Hadamard Conjecture for connected regions. 
Furthermore, it was known that the conjecture holds for the
regions consisting of geodesic balls of fixed radius in a Cartan--Hadamard manifold.
Finally, connectedness of isoperimetric regions can be established in many cases.
For example in a homogeneous space, a component of a disconnected isoperimetric region can be 
translated by an isometry until it touches a second component without overlapping, 
contradicting the known boundary regularity properties of isoperimetric regions. 
Nevertheless connectivity does not hold in general Cartan--Hadamard manifolds.

The existence of isoperimetric regions in Cartan--Hadamard manifolds is not automatic, since 
a sequence of increasingly efficient regions enclosing a 
given volume may drift off to infinity, yielding no limiting region.
The manifolds that we construct  are isometric with
hyperbolic space away from a compact set, and existence holds in this setting.
A general existence result that covers our examples is given
by Ritore and Sinestrari  \cite[Theorem 1.21]{RitoreSinestrari}.

In Section~\ref{dimensiontwo} we establish Theorem~\ref{nonconnected} in dimension two.
In Section~\ref{dimensionthree} we present additional arguments needed to extend the construction to three
dimensions. We end with some concluding remarks in Section~\ref{remarks}.

\section{Dimension Two} \label{dimensiontwo}
In this section we construct an isoperimetric region in a two-dimensional Cartan--Hadamard manifold that
consists of two components.  A similar construction gives an isoperimetric region possessing any
finite number of components.

\begin{theorem} \label{dim2}
There are  two-dimensional Cartan--Hadamard manifolds that contain
disconnected  isoperimetric regions.
\end{theorem}

\subsection{Curvature in Normal Coordinates} \label{constructPQ}
 
A metric described in geodesic normal (polar) coordinates    $(r, \theta)$  on $\RR^2$  by
 $$
 g = dr \otimes dr + h(r)^2 d\theta \otimes d\theta, 
 $$
has Gauss curvature given by
\begin{eqnarray}
 K(r,\theta) =- \frac{ h''(r)}{h(r)} . \label{eqn2}
\end{eqnarray}
A derivation can be found in Spivak \cite[p. 136] {Spivak}.

The function $h(r) = r$ results in the flat Euclidean plane $\EE^2$, while choosing $h(r) = \sinh (r)$ gives the hyperbolic metric on $\HH^2$, with
curvature $ K(r,\theta) = -1$.  
In these coordinates the length of the circle of radius $r$ is $2 \pi h(r)$.
The first variation of arc length formula then shows that  the normal curvature
of the radius $r$ circle is $k(r) = h'(r)/h(r)$.
Fix a positive constant $R > 1$   and let 
$$
 C = R - \sinh^{-1} R  \mbox{   \  \  and   \ \  }   \delta =\frac{1}{4 \pi R \sinh (2R) }.
$$

\begin{lemma} ~\label{h(r)}
There is a smooth function $h : \RR^+ \to \RR$   
with the following properties:
\begin{enumerate}
\item  $h(r) = r $ for $ 0 \le r \le R - \delta $,
\item  $h(r) = \sinh (r - C)$ for $ r>R +\delta $,
\item  $h(r) $ is convex for all $r>0$.
\end{enumerate}
\end{lemma} 
\begin{proof}
The graphs of $y=x$ and $y = \sinh (r - C)$ cross at $(R,R)$.
For $R>1$ we have $\delta <R$, and the two graphs
can be smoothly interpolated by an increasing convex function over the 
interval $(R-\delta, R+\delta)$.
The  graph of  the resulting function $h$ is shown in Figure~\ref{hgraph}.
\end{proof}
 
\begin{figure}[htbp] %  figure placement: here, top, bottom, or page
   \centering
   \includegraphics[width=2in]{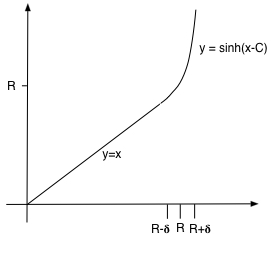} 
   \caption{Graph of the radial function $h$.}
   \label{hgraph}
\end{figure}

 We use Lemma~\ref{h(r)} to  explicitly construct a 2-dimensional Cartan-Hadamard manifold $P(R)$. We describe
 the Riemannian metric on $P(R)$ using normal coordinates around the origin.
The manifold $P(R)$ is the simply connected surface with smooth metric  
 $$
 g = dr \otimes dr + h(r)^2 d\theta \otimes d\theta.
 $$
 
\begin{lemma} ~\label{nonpositive}
The 2-dimensional manifold $P(R)$ has nonpositive curvature.
 It is flat on the disk of radius $R - \delta$ about the origin and
the exterior of the disk  of radius $R+\delta$ is isometric to the 
exterior of a disk of radius $C+\delta$ in the hyperbolic plane $H^2$.
On the annulus  with radii $R - \delta$ and $ R + \delta$, $P(R)$ has negative curvature, and the area of this annulus is less than $1/R$.
 \end{lemma}
\begin{proof}
The curvature properties follow from Equation~(\ref{eqn2}).
Note that $\delta < R$ and therefore Fubini's Theorem gives an area bound for the annulus with radii $R - \delta$ and $ R + \delta$ of
$$ 
2\pi  \sinh (R + \delta - C) (2\delta) = 4\pi  \sinh (  \delta +\sinh^{-1} R ) ( \delta)  < 4\pi \sinh (2R)  (\delta ) = 1/R.
$$
 \end{proof}
 %Note that if we take $\delta <1/( 4 \pi  \sinh (2R))$ then the annular region has area less than 1.

The metrics  on the surface $P(R)$ outside a disk of radius $R+\delta$ and on the hyperbolic plane outside a hyperbolic disk of radius $C+\delta$ agree.
as indicated in Figure~\ref{fig:P}. 
Thus $P(R)$ can be obtained from the hyperbolic plane by changing the hyperbolic metric on a disk of radius $C+\delta$.
Given a constant $d>0$, we can change the metric on the hyperbolic plane in the same way on two subdisks 
$D_1$ and $D_2$ of $H^2$ whose distance from one another is $d$.
We call the resulting surface $Q(R,d)$.

\begin{figure}[htbp] %  figure placement: here, top, bottom, or page
\centering
\includegraphics[width=1in]{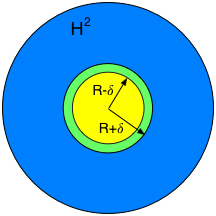} \hspace{1in}   \includegraphics[width=1in]{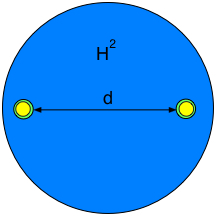} 
\caption{The Cartan-Hadamard surface $P(R)$ at left has a rotationally symmetric nonpositively curved Riemannian metric. It is isometric to hyperbolic 2-space in the complement of a disk $D$. The Cartan-Hadamard surface $Q(R,d)$ at right is isometric to $\HH^2$  in the complement of two  disks $D_1$ and $D_2$. Each disk is isometric to $D$ and they are separated by distance $d$.}
\label{fig:P}
\end{figure}

\subsection{Isoperimetric Regions in $P$ and $Q$}
We now show that for suitable choices of $R$ and $d$, there is an area $A$ such that the isoperimetric region for area $A$ in  $Q(R,d)$ 
is disconnected. 
For simplicity we use explicit values for $R$ and $d$, taking the values $R=100$ and $ d=1000$ to construct $Q=Q(100,1000)$, 
and investigate the isoperimetric region in $Q$ with area $A = 62830$.

For these values we have  
$$
C = R - \sinh^{-1} (R) \approx 94.70.
$$
and 
$$ 
\delta =\frac{1}{4 \pi R \sinh (200) }  \approx    (2.2) \times 10^{-90}.
$$
By Lemma~\ref{nonpositive}, the area of each of the two annuli in $Q$ that transition from flat to hyperbolic metrics is less than .01.
This is sufficiently small to ensure that these annuli will not play a significant role in the area comparisons we will make below.

A Euclidean disk of radius $ 100$ has area $ A_0 \approx  31416$ and circumference $\approx 628$,
and a hyperbolic disk of radius $100 - C \approx 5.3$ has area   $A_1 \approx  629$ and circumference $\approx 628$.
Thus $A$ is slightly smaller than the area enclosed in two  Euclidean disks of radius 100. 
For comparison, a hyperbolic disk with area $A_0$ has
radius $  \approx 9.9$, and circumference $ \approx 62838$.

Now consider an isoperimetric region of area $A = 62830$ in $Q(100, 1000)$.
One possibility for such a region consists of  a disconnected region $W_0$ containing  
two disjoint disks, with one contained in each of the two flat subdisks $D_1$ and $D_2$ of $Q$.
We will show that this $W_0$ has shorter boundary than any connected  region with the same area.

\begin{lemma} ~\label{1257}
A connected isoperimetric region $W$ in $Q$ has $|\partial W| > 1257$.
\end{lemma}
\begin{proof}
We consider first the case  that $W$ is disjoint from  $(D_1 \cup A_1)  \cup ( D_2 \cup A_2) $.
Then $W$ is contained in a subsurface of $Q$ that is isometric to $\HH^2$. The 
length of the boundary of $W$ can then be bounded below by applying the
 isoperimetric inequality in $\HH^2$ which states that a region enclosing 
 an area of $A_0$  has boundary length greater than $62837 >  1257$.
 
Suppose now that  $W$ intersects  $D_1 \cup A_1$ and is disjoint from $D_2\cup A_2$ .
The area of $D_1$ is approximately $A/2$, so
that $|W \cap (D_1 \cup A_1) |< 2A/3$ and $W$ intersects the complement of  $D_1 \cup A_1$ in
a region of area at least $A/3 =20610$. 

The question of how short can one make the boundary curve of a region in hyperbolic space that lies
in the complement of a convex set (the complement of $D_1 \cup A_1$) was  studied by  Choe and Ritore \cite{ChoeRitore:07}.
They show that in two and three-dimensional hyperbolic space, a region in the complement of a convex domain  
 has longer boundary length then a half-disk of the same area. 
 A half-disk in a half-space of $\HH^2$ with area 20610 has boundary length greater than 20606.
As a result, a connected  region $W$ that meets $D_1$ must have boundary length no less than 20606 in the complement of $D_1$.
Since $20606 > 1257 $, the Lemma holds for such regions.

Finally assume that a region  $W$ meets  both $D_1 \cup A_1$ and $D_2\cup A_2$. Since the two disks are
distance 1000 apart, we must have $|\partial W| > 2000$ for $W$ connected. Again $2000 > 1257 $ and
a connected region that intersects both $D_1$ and $ D_2$  has $|\partial W| > 1257$. 
\end{proof} 

\noindent
{\em Proof of Theorem~\ref{dim2}.  }
The surface  $Q$ contains two disjoint flat disks $D_1$, $D_2$ with
combined area $ 62830$. Each disk  is
 surrounded by a radius $2\delta$ annulus of
area less than .01, which we call $A_1$ and $A_2$ respectively.  Outside these annuli,  
$Q$ is isometric to the complement in hyperbolic space of two round disks.
The union of these two disks forms a surface $W_0$ with $|\partial W_0| < 1256$

Now let $W$ be an isoperimetric region having area $A$ and assume for contradiction that $W$ is connected.
By Lemma~\ref{1257} we have $|\partial W|  > 1257$.
Since $|\partial W_0| < 1256$,  we conclude that  an isoperimetric region of  area $A$
cannot be connected.
\qed

\section{Dimension Three} \label{dimensionthree}

In this section we extend the  two-dimensional arguments of Section~\ref{dimensiontwo} to
construct a disconnected isoperimetric region in a three-dimensional Cartan--Hadamard manifold.

\begin{theorem} \label{dim3}
There are  three-dimensional Cartan--Hadamard manifolds that contain
disconnected  isoperimetric regions.
\end{theorem}
   
We start by constructing for each $R>0$ a  Cartan-Hadamard 3-manifold $P_3(R)$ with
a  spherically symmetric Riemannian metric. The metric
 in geodesic normal coordinates is given by
 $$
 g = dr \otimes dr + h(r)^2  \omega  
 $$
where $ \omega  $ is the Riemannian metric on the unit 2-sphere and $h(r)$ was given in Section~\ref{dimensiontwo}.
This gives a Riemannian manifold $(M, g)$
that is symmetric under rotations around the origin. 
At each point  $p \in M$ we can construct an orthonormal frame with one vector in the radial direction $\partial/ \partial r$ and two additional
 vectors
tangent to the sphere $S(r)$ with $r = r(p)$.  The sectional curvatures in the tangent planes containing the vector
$\partial/ \partial r$ are determined by the totally geodesic plane through $p$ and the origin, and coincide with those
computed for  $P(R)$ in Section~\ref{dimensiontwo}. Thus these sectional curvatures are nonpositive.
It remains to compute the sectional curvature $K_S$ corresponding to a 2-plane tangent to the 
sphere $S(r)$.  

We determine this with the Gauss Equation, which relates the ambient sectional curvature $K_S$ to the intrinsic curvature $K_r$ of the 2-sphere
$S(r)$ and the determinant of the second fundamental form of $S(r)$:
$$
K_r = K_S + det(II)
$$
The constant curvature sphere  $S(r)$  has area $4\pi h(r)^2$, so the Gauss-Bonnet Theorem implies that $K_r = 1/h(r)^2$.
The principle curvature $k_n(r)$ at a point on  $S(r)$ can then be determined from the first variation
formula for the length $L(r) = 2\pi h(r)$ of a great circle, which states that $L'(r) = L(r) k_n(r)$.
This implies that $k_n(r) = L'(r)/L(r) = h'(r)/h(r)$.  
So $det(II) = h'(r)^2/h(r)^2$, and
$$
K_S = K_r - det(II) =  \frac{1- h'(r)^2 }{ h(r)^2  }   .
$$
Since $h'(r) \ge 1$ we see that $K_S$  is nonpositive.
Note that $K_S = 0 $  when $h(r)=r$, and that  $K_S = -1 $ when $h(r) = \sinh (r - C)$, as expected.
 
Having shown that  $(M, g)$ is a Cartan--Hadamard manifold, we construct  $P_3(100)$ and  $Q_3(100,d)$ as before, but with
a different value fo $\delta$.
The manifold $Q_3(100,d)$ contains two flat balls $B_1, B_2$ of radius 100, surrounded by thin collars  $C_1$ and $C_2$ of radius $2\delta$.
The exterior  $(B_1 \cup C_1) \cup (B_2 \cup C_2) $ is hyperbolic, with the two balls separated by a distance $d$.
We  take $d$ to be a large distance to be fixed later. For the value  $R=100$, the volume enclosed by a Euclidean sphere
of radius $R$ is $ \approx 4188790$ and the area of the enclosing sphere is $A \approx 125664$.  
We choose $\delta $ sufficiently small so that collars  $C_1$ and $C_2$ around $B_1, B_2$ each have volume less than .01.

We now look for an isoperimetric region $W$ in $Q^3(100,d)$ that encloses volume  $V = 8377580$, approximately  the volume of $B_1 \cup B_2$.  
We first note that an isoperimetric region exists, since  $Q_3(100,d)$ is isometric to $\HH^3$ outside a compact set,  
\cite[Theorem 1.21]{RitoreSinestrari}.
Again we consider three cases, according to the three possible ways that a region $W$ can intersect $B_1 \cup  B_2$.

One candidate for an isoperimetric region enclosing volume  $ 8377580 $ is $W_0$,  which consists of two flat balls in each of $B_1, B_2$. 
Its boundary area satisfies  $|\partial W_0|   <  251229$. We will show that this is less than the boundary area of any competing connected region.

We first consider a region  $W \subset Q_3(100,d)$ enclosing volume $V$ and disjoint from both $B_1 \cup C_1$ and $ B_2 \cup C_2$. 
Thus $W$ is isometric to a subset of hyperbolic space and we can bound its boundary area using the
isoperimetric inequality for $\HH^3$.
In hyperbolic 3-space, the surface area of a sphere of radius $r$ is
$4 \pi  \sinh^2 (r ) = 2\pi (\cosh (2r) -1)$ and the volume of a radius $r$  ball is 
$   \pi   \sinh  (2r) - 2 \pi r$.  
Using this, we see that  the volume of a ball of radius 7.74475
is approximately 8377580, and the boundary area of this ball is 
$ |\partial E|  >  1.67 \times 10^7  $ which is larger than $  |\partial W_0| \approx 251327$.
So an isoperimetric region cannot be disjoint from both $B_1 \cup C_1$ and $ B_2 \cup C_2$.

Now suppose that a connected isoperimetric region $W$ meets just one of  $B_1 \cup C_1$ and $ B_2 \cup C_2$, say
$B_1 \cup C_1$ , and is disjoint from the second. 
Since the volume of $B_1 \cup C_1$  is roughly half the volume enclosed by $W$, we must have that
$W$ intersects the complement of $B_1 \cup C_1$ in a region
whose volume is greater than $V/3 >2792526 $.
We now apply a result of Choe and Ritore, who established  isoperimetric inequalities
for regions in the complement of a convex region in hyperbolic space.
These give a lower bound on the area  $|\partial W|$ in the complement of  $B_1 \cup C_1$.
Choe and Ritore show  that any region in the complement of a convex region in $\HH^3$ that encloses volume $V/3$ has
boundary area no less than a hemisphere enclosing the same volume of $V/3$ \cite[Theorem 3.2]{ChoeRitore:07}. 
In particular, a computation shows that the area of such a boundary surface 
is at least $5.585 \times 10^6 $ which again is larger than $ |\partial W_0|  \approx 251327 $.
Thus a connected isoperimetric region cannot be disjoint from $ B_2 \cup C_2$.  

In the remaining case a connected isoperimetric region  $W$ meets both  $B_1 \cup C_1$ and $ B_2 \cup C_2$. 
We  show that for $d$ sufficiently large, the  boundary of such an isoperimetric region must have  area $|\partial W| > 2.6 \times 10^5 $,
eliminating it as a possible isoperimetric region.
Note that it is not true that any connected region meeting both $B_1 \cup C_1$ and $ B_2 \cup C_2$ has very large area, since
it is possible to span the region between these balls by a thin tube of arbitrarily low area.
Thus we cannot say that $|\partial W|$ is large simply because $d$ is large, as in dimension two.
However we know that the boundary surface $\partial W$ of an isoperimetric region 
is a stable constant mean curvature, so that there is a tubular neighborhood of $\partial W$  that contains
no nearby surface of less area that encloses the same volume.
We now show that the area of any connected stable constant mean curvature surface that meets both  $B_1$ and $B_2$
goes to infinity with $d$, and in particular that for such $W$, $|\partial W| > 251328$  for large $d$.

Suppose that $W$ is an isoperimetric region  that  meets both   $(B_1 \cup C_1$) and $( B_2 \cup C_2)$.
Then $\partial W$  is a stable constant mean curvature surface in $\HH^3$.
Let $\gamma \subset Q$ be a geodesic connecting $\partial C_1$ and $\partial C_2$.
Then there is a family of totally geodesic planes $H_t, ~ 0 \le t \le d$ in $Q$ that are perpendicular to $\gamma$
and separate $B_1$ from $B_2$.
For constants $0 \le a < b \le d$ define the region $I_{[a, b]}$ to be the submanifold of $Q$
between $H_{a}$ and $H_{b}$, isometric to  a submanifold of hyperbolic space bounded by
two hyperbolic planes at distance $b-a$. We consider the intersection of the 
stable constant mean curvature surface $\partial W$ with $I_{[a, b]}$.  

We claim that there is a constant $C$ such that  the second fundamental form of  $\partial W \cap I_{[d/4, 3d/4]}$ 
is bounded above by $C$ for all $d \ge 4$. Our argument follows an argument given by
Meeks, Perez and Ros  \cite[Theorem 2.16]{MeeksPerezRos}.  
If there is a sequence of values $d_i$ and a sequence
of isoperimetric regions $W_i$ with no upper bound  for the second fundamental form on the surfaces
 $\partial W_i \cap  I_{[d_i/4, 3d_i/4]}$, then a sequence of blow ups, or rescalings, at  a point of increasing
second fundamental form yields a sequence of stable constant mean curvature surfaces
with second fundamental form having norm one at that point. The limit of such a sequence of
rescalings gives a complete stable minimal surface
in $\RR^3$ that is not flat. However a
theorem of do Carmo and Peng  \cite{doCarmoPeng}, and  independently 
Fischer-Colbrie and Schoen \cite{FischerColbrieSchoen}, shows that the plane is the only stable
minimal surface in $\RR^3$. Thus the second fundamental form of  $\partial W \cap I_{[d/4, 3d/4]}$ 
is uniformly bounded above by some constant $C$, independent of $d$.

For $w \in \HH^3$, let $B(w,\epsilon)$ denote the ball about $w$ of radius $\epsilon$.
An upper bound on the second fundamental form of $\partial W$ implies
 a uniform lower bound on Area$(\partial W \cap B(w,\epsilon))$ for $w \in \partial W$ and for
some  fixed small $\epsilon>0$.  Indeed, for small enough
 $\epsilon$,  $ B(w,\epsilon)$ meets $\partial W$ in a topological disk containing $w$,
and the area of such a disk  is at least that of a radius $\epsilon$ hyperbolic disk, or
$2\pi (\cosh (\epsilon) -1).$  

If we take 
$$
d >  \frac{251328(4 \epsilon) }{   \pi (\cosh (\epsilon) -1)  } ,
$$
then there is enough room in $ I_{[d/4, 3d/4]}$ 
for $251328 /(4\epsilon)$ disjoint balls centered on points of $\partial W$,
each intersecting   $\partial W$ in an area of no less than $2\pi (\cosh (\epsilon) -1).$ 
We conclude that the area of $\partial W$ is at least $251328$ for such $d$ and that
$W$ cannot be an isoperimetric region enclosing volume $V$.
Thus we have eliminated the possibility of a connected isoperimetric region in $Q(d)$ for $d$ large.  
We conclude that the isoperimetric region in $Q(d)$ for volume $V=8377580$ is not connected for sufficiently large values of $d$.
\qed

\section{Remarks} \label{remarks}

The constructions we gave for dimension two and three can be extended to all higher dimensions, and
it is likely that the resulting Cartan--Hadamard manifolds have similar disconnected isoperimetric regions.
However the arguments we gave rely on results known only in two and  three dimensions.  In particular,
we do not know the analogs of the results of Choe and Ritore \cite{ChoeRitore:07},
do Carmo and Peng  \cite{doCarmoPeng} or
Fischer-Colbrie and Schoen in higher dimensions.

While disconnected regions are certainly not convex, it is natural to ask whether each connected component
of an isoperimetric regions is convex.
In a two-dimensional Cartan--Hadamard manifold, this is indeed the case, and
each component of an isoperimetric region is convex. 
 
\begin{lemma} ~\label{convex}
Each component of an isoperimetric region $W$ in a two-dimensional Cartan--Hadamard manifold is a convex disk.
 \end{lemma}
\begin{proof}
Suppose not.  Then there is a geodesic arc $\gamma$ from $a \in \partial W$ to $b \in \partial W$
whose interior lies outside of $W$.  
Deforming $ \partial W$ inwards slightly 
at a point disjoint from the arc between
$a,b$ results in a shortening of length that decreases the enclosed area.
Replacing the adjacent arc of $\partial W$ between $a,b$ with
$\gamma$ gives a shorter curve enclosing more area.  By sliding the point $b$ closer to $a$ we can
arrange that the increase in area matches the decrease from the first deformation, resulting in a
shorter curve enclosing an equal  area. This contradicts our assumption that $W$ is an isoperimetric region. 
 \end{proof}
 
Nonconvex isoperimetric regions appear to exist in three dimensions, though we have not constructed an explicit example.
One plausible approach to construction of a nonconvex isoperimetric region follows the methods used to prove Theorem~\ref{dim3}.
Namely, we move $B_1$ and $B_2$ closer together, decreasing $d$, until a connected region
becomes more efficient at enclosing volume $V$ than disconnected regions with components in 
each of $B_1$ and $B_2$.  Note that this happens for some positive $d$.  
As in Theorem~\ref{dim3} an isoperimetric region $W$ will necessarily contain large subregions in
each of $B_1$ and $B_2$.  These will have a boundary surface with mean curvature
close to zero.
Convexity implies that these regions are connected by a surface that spans the hyperbolic
region between  $B_1$ and $B_2$.  If we look at the intersection of this surface
with the hyperbolic plane separating $B_1$ and $B_2$ then we see that its intersection
with this plane is a convex curve whose  mean curvature is greater than 1/2.
Since $\partial W$ has constant mean curvature, this gives a contradiction to the assumption of
convexity.  This suggests that convexity is unlikely to hold for connected isoperimetric regions.

The examples we constructed were flat on two balls and non-positively curved everywhere.
It is straightforward to make a small perturbation of the constructed metrics to give 
examples of disconnected isoperimetric regions on manifolds with strictly negative curvature.
One simply perturbs the function $h(r)$ of Section~\ref{dimensiontwo} slightly so that it is everywhere
strictly convex.

\end{document}